\documentclass[12pt]{article}
\title{The spectrum of the random environment \\ and localization of noise }
\author{Dimitris Cheliotis\\ TU Eindhoven, \sc Eurandom \and B\'alint Vir\'ag \\University of Toronto}
\date{\today}

\usepackage{amsfonts}
\usepackage{graphicx}
\usepackage{amsmath}
\usepackage{amsthm}
\usepackage{natbib}

\usepackage{amssymb,latexsym}

\theoremstyle{plain}
    \newtheorem{theorem}{Theorem}
    \newtheorem{lemma}[theorem]{Lemma}
    \newtheorem{question}{Question}
    
    \newtheorem{proposition}[theorem]{Proposition}

\theoremstyle{definition} % For roman text in the body

    \newtheorem{remark}[theorem]{Remark}

\theoremstyle{remark} % For an italic header, more subtle than definition style

\newcommand{\disdis}{{\mathcal Y}}
\newcommand\snorm[1]{\|#1\|_*}

\newcommand\anp{{\mathcal X}}
\newcommand\mnote[1]{} %off
\newcommand\be{\begin{equation}}

\newcommand\ee{\end{equation}}

\newcommand{\BBn}[1]{\mathbb{#1}}
\newcommand{\C}[1]{\mathcal{#1}}
\newcommand{\comment}[1]{}
\newcommand{\eps}{\varepsilon}

\newcommand{\R}{{\mathbb R}}

\newcommand{\ev}{\mbox{\bf E}}

\newcommand{\dist}{\mbox{\rm dist}}

\newcommand{\Var}{\operatorname{Var}}

\newcommand{\sm}{{\raise0.3ex\hbox{$\scriptstyle \setminus$}}}

\newcommand{\Tr}{\operatorname{Tr}}

\newcommand{\stab}{\operatorname{Stab}}

\newcommand{\gl}{\lambda}

\newcommand{\gb}{\beta}
\newcommand{\gep}{\varepsilon}
\newcommand{\gd}{\delta}

\newcommand{\EE}{{\bf E}}

\newcommand{\BB}[1]{\mathbb{#1}}
\newcommand{\G}{\mathcal{G}}

\usepackage{natbib}
\begin{document}

\maketitle
\begin{abstract}
We consider random walk on a mildly random environment on finite
transitive $d$-regular graphs of increasing girth. After scaling and
centering, the analytic spectrum of the transition matrix converges
in distribution to a Gaussian noise.  An interesting phenomenon
occurs at  $d=2$: as the limit graph changes from a regular tree to
the integers, the noise becomes localized.
\end{abstract}

\begin{figure}[h]
\centering
\includegraphics[height=1.8in]{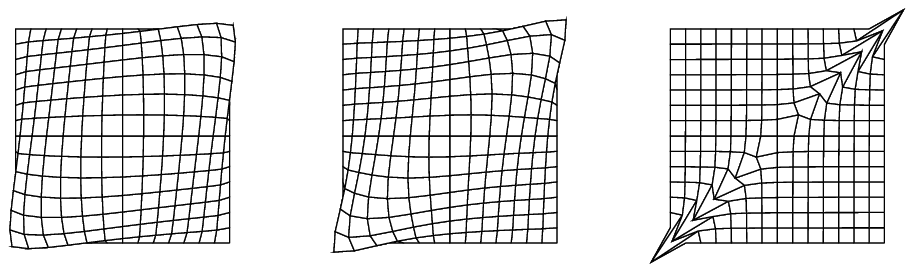}
\end{figure}

\section{Introduction}

Localization phenomena for eigenvalues of random media have
received considerable attention lately. Among the several
results, we point out a specific one. If one perturbs the
Laplacian of the nearest-neighbor graph on the integers by
a very small i.i.d.\ potential, its spectrum becomes pure
point immediately. In contrast, the same change on a higher
degree regular tree preserves the continuous spectrum, see
\cite{klein}.

The results of this paper point to an analogous, but distinct
 phenomenon. While the above localization phenomena are due
to large-scale behavior of the eigenvectors, we find a similar
dichotomy for the local behavior.

For this, we consider sequences of vertex-transitive finite graphs
$G_n$ with degree $d\ge 2$ and increasing girth (the length of the
shortest simple cycle); these converge locally to either the
integers ($d=2$) or a regular tree of higher degree ($d\ge 3$). On
these graphs, we consider a perturbation of the Laplacian, given
by the well-studied random walk on random environment
model with small noise (see, for example \cite{Z}. The noise in the Laplacian creates a
Gaussian noise in the spectrum. Our main discovery is that in the
limit this noise is local for the integers and has long-range
correlations for higher-degree trees.

We first consider more general graphs $G$. We will have the
following standing assumptions.

\bigskip

{\bf Assumptions.} $G$ is a vertex-transitive graph of
finite degree. Let $M$ denote the transition probability
matrix for nearest-neighbor simple random walk on $G$. For
the random environment on $G$ one randomly modifies
transition probabilities along edges. Its transition matrix
is defined as $M+\eps B$, where $B$ is a random matrix
satisfying the following.

The random variables $B_{u,v}$ have mean zero and variance
1. Further, the $B_{u,v}$'s corresponding to different $u$
are independent, and for some constant $c_1$ and all
vertices $u,v \in G$ we have
\begin{eqnarray}\label{boundedPert}
|B_{u, v}|&\le& c_1 M_{u, v}, \\
\label{zerosum}\sum_{g\in \stab(u)} B_{u,gv}&=&0.
\end{eqnarray}
Here  $\stab(u)$ is the stabilizer of $u$ in the
automorphism group of $G$. Condition \eqref{zerosum} means
that the sum of the bias of the random environment over
symmetric directions is zero, which is a bit stronger than
just requiring $M+\eps B$ to be stochastic. Condition
\eqref{boundedPert} implies that $B_{u,v}$ vanishes when
$uv$ is not an edge and that for small $\eps$ the entries
of $B+\eps M$ are in $[0,1]$. These assumptions will be in
effect for the rest of the paper.
\bigskip

For a finite $G$, let $\mu_{\eps}$ denote the empirical
probability measure of the eigenvalues of $M+\eps B$. Then
as $\eps \to 0$ we will show an expansion
$$
\mu_{\eps}=\mu_0+\frac{1}{2}\mu'' \eps^2  + o(\eps^2),
$$
where $\mu''$ is a random functionals. We define the  second difference quotient
$$
m_{\eps}=|G|^{1/2} \frac{\mu_{\eps}-\mu_{0}}{\eps^2},
$$
where the extra factor upfront makes the scaling consistent
as $G$ changes. This is the centered and scaled empirical
eigenvalue measure of $M+\eps B$. Let $\mathcal X$ denote
the space of complex functions analytic in a neighborhood of the closed unit disk. 
We  call $f$ real if it maps $\R$ to $\R$.
In Proposition \ref{epsto0}, we
will show that for $f\in \mathcal X$ the limit
\begin{equation}\label{eepsto0}
T_G(f) = \lim_{\eps \to 0} \int f \,d m_{\eps}
\end{equation}
exists. Our first theorem identifies the covariance
structure of  $T_G(f)$ as $f$ ranges over real functions in $\anp$.

For a possibly infinite vertex-transitive graph $G$ and complex $|\lambda|<1$, let
  $$
  p_\lambda(x)=1/(1-\lambda x),\qquad
  \mathcal G_\lambda=p_\lambda(M)=(I-\lambda M)^{-1},
  $$
the Green's function  corresponding to $G$. Let
$o$ be a marked vertex of $G$  and define
 \begin{equation}\label{H}
H_G(p_\lambda,p_\mu):=\frac{\gl \mu}{2}\,
\partial_{\gl ,\mu}^2\left(\gl^2\mu^2 \sum_{v\not=o} \ev\left[
(B\G_\lambda)_{ov}(B\G_\mu)_{ov}\right]\ev\left[(B\G_\lambda)_{vo}(B\G_\mu)_{vo}\right]
\Big] \right).
 \end{equation}
We will show in Lemma \ref{uec} and Proposition
\ref{hplemma} that $H_G(p_\lambda,p_\mu)$ is well-defined
and extends uniquely to a bilinear form $H_G$ on $\mathcal
X$.

\begin{theorem}\label{t2}
For any finite vertex-transitive graph $G$ and $f,g\in
\mathcal X$ we have
 $$
 \ev T_G(f)=0, \qquad
 \ev [T_G(f)T_G(g)]=H_G(f,g).
 $$
\end{theorem}
We now consider sequences of transitive graphs $G_n \to G$ locally,
which means that for every $r$ the $r$-neighborhood of a
fixed vertex of $G_n$ eventually agrees with that in $G$.
\begin{theorem}[Normality] \label{normality}
Let $G_n\to G$, an infinite graph, and let $G_n$, $G$
satisfy the assumptions above with uniform constant $c_1$.
 Then as $n\to\infty$,
jointly for all real $f\in \mathcal X$ the random variables
$T_{G_n}(f)$ converge weakly to mean zero normal random
variables $T(f)$. Moreover, for all $f,g\in \mathcal X$ we
have
$$
\ev\, [T_{G_n} (f)T_{G_n}(g)] \to \ev \,[T(f)T(g)] =
H_G(f,g).
$$
\end{theorem}
Finally, we consider the case when $G_n\to \BB{T}_d$, the
$d$-regular tree. Since $\mathbb T_d$ is bipartite, we
expect that the random functionals $m_{\eps}$ have a
symmetric limit. Therefore it is natural to consider $T_n$
for even functions of the form $f\circ s$ where $s(x)=x^2$
and $f\in \mathcal X$. Note that for any finite graph $G$
$$
T_{G}(f\circ s) = \tilde T_{G}(f),
$$
where  $\tilde T_G(f)$ is defined analogously to $T_G(f)$
but in terms of $\tilde \mu_{\eps}$, the empirical
eigenvalue measure of $(M+\eps B)^2$.

It is known (see \cite{LW}, Theorem
13.4) that $\tilde \mu_n$ corresponding to $G_n$ converges weakly to the measure with
density
\begin{equation}
\label{limESD} a_d(x):=\frac{2d^2}{\pi}\frac{\sqrt{x(\rho-x)}}{1-x}
\,{\bf 1}_{0<x<\rho},
\end{equation}
where $\rho=\rho_d:=4(d-1)/d^2$, the squared spectral
radius of the walk on the infinite $d$-regular tree
$\mathbb T_d$. Our next theorem gives an explicit
expression for the limiting covariance structure.

\begin{theorem}[Tree limits] \label{thm1}\label{main}
Assume that $G_n \to \mathbb T_d$, the $d$-regular tree.
Then jointly for all real $f\in \mathcal X$ we have the
convergence in distribution $ \tilde T_{G_n}(f)\to \tilde
T(f)$, where the $\tilde T(f)$ are jointly normal and have
mean 0. Moreover,
\begin{equation}\label{covariance}
\ev[\tilde T(f) \tilde T(g)]=\int_0^{\rho} \int_0^{\rho}
f'(x) g'(y) \,\gb_d(x,y) \,dx dy,
\end{equation}
where for $d>2$ the kernel $\gb_d$ is given by
\begin{equation}\label{kernel} \beta_d(x,y)=\frac{2d^4}{\pi^2}
\frac{(d-2)\,\kappa(x)\kappa(y) } {16(2d-3)(x-y)^2+(d-2)^2
A(d,x,y)},
\end{equation}
with the semicircle function
\begin{eqnarray*}
\kappa(x)&:=&2d\sqrt{x(\rho-x)} \qquad \mbox{and}\\
A(d,x,y)&:=&\rho
\kappa\big(\frac{x+y}{2}\big)^2+4(d+3)(x-y)^2+\rho^2
(d-2)^2.
\end{eqnarray*}
For $d=2$, the covariance is given by
$$
\frac{32}{\pi}\int_0^1 f'(x) g'(x) \sqrt{x(1-x)} dx,$$
which corresponds to
$\gb_2(x,y)=\frac{32}{\pi}\sqrt{x(1-x)} \ \gd_x(y).$
\end{theorem}
The kernel function $\beta_d$ becomes singular as
$d\downarrow 2$. This can be seen in the figure on the
first page showing graphs $\beta_d:[0, \rho]^2\to\R$ for
$d=4,3,2.1$ from above. In words, in the $d=2$ case the
centered and scaled empirical measure converges to the
so-called $H_1$-noise with a certain density. In contrast
with the $H_{1/2}$-noise that arises in the limit of the
Gaussian Orthogonal Ensemble and Haar unitary random matrix
models (among others; see, for example \cite{DE01},
\cite{ZA} and the references therein), this noise is local,
as the $\delta$-function in the covariance formula shows.
$H_1$ noise appears typically for complex eigenvalues, see
\cite{rv}. Formula \eqref{covariance} is written in a form
to show how the kernel becomes singular (localized) as
$d\downarrow 2$. Recall that $\pi^{-1}\gep/(x^2+\gep^2)$
converges as $\gep\to 0$ to the delta function at $x=0$.

The presence of $f'$, $g'$ in the formula for the
covariance is a priori expected. When the matrix is
randomly perturbed, the eigenvalues move in random
directions. For every eigenvalue $\lambda$, the function
$f(\lambda)$ changes by a small amount proportional to $f'$
times the change in $\lambda$. However, the picture here is
more complicated because the perturbation is of order
$\eps^2$. The kernel $\beta$ reflects the correlation
between the random change in eigenvalues as well as their
density.

Theorem \ref{main} motivates the following question:

\begin{question}
Which natural sequences of graphs give rise to localized noise?
\end{question}

Our strategy is as follows. To prove normality, we use a
central limit theorem for dependent variables based on
Stein's method, rather than the usual method of moment
computations. For computing the covariance, we consider
traces, and so we will have to count certain paths. We use
the Green function to jointly treat paths of different
lengths and to avoid complicated computations with
orthogonal polynomials that normally arise in this context.
Theorems \ref{t2}, \ref{normality} and \ref{main}  are
proved in Sections \ref{s.covariance}, \ref{s.normality}
and \ref{s.trees} respectively.

%\noindent \textbf{Note:} In the following, we will denote the
%double integral in \eqref{covariance} by $H_d(f,g)$.

The finite graphs that we study arise naturally. The
simplest way to pick an interesting transitive graph is to
consider a Cayley graph of a sufficiently complicated
finite group. As it is discussed in \cite{GH}, such Cayley
graphs will typically have girth tending to infinity; for
example, random Cayley graphs of simple groups of
increasing order will have this property. Of course, for
$d=2$, the only examples are graphs consisting of cycles.

Random perturbations of eigenvalues have been extensively
studied in the literature. Reference \cite{VP} studies
properties of the eigenvalues of a random perturbation of a
fixed matrix without the restriction for stochasticity or
positivity of the matrices involved (see Sections 1.3 and
2.2). Another possible approach to our problem is via
perturbation expansions for eigenvalues (see for example
\cite{DE} Sections 6.3, 6.4), but this requires more
control of the eigenvectors.

\section{The $\eps\to 0$ limit} \label{model}

We first show that for finite graphs $G$ the limit
\eqref{eepsto0} defining $T_G(f)$ exists and identify it.
We will rely on the Assumptions about the perturbation
matrix $B$, in particular, \eqref{zerosum}. In fact,
throughout this paper, we only need a simple consequence of
\eqref{zerosum}. Namely, for all $k\ge 1$ and $v\in G$, we
have
\begin{equation}
(BM^k)_{v, v}=0.  \label{zeroentry}
\end{equation}
Indeed, for any $g\in \stab(v)$ this equals
$$\sum_{w\in G} B_{v,w}M^k_{w,v}=\sum_{w\in G} B_{gv,gw}
M^k_{gw,gv} = \sum_{w\in G} B_{v,gw}
M^k_{w,v}.
$$
Averaging each term over all $g\in \stab(v)$ and using
\eqref{zerosum} we get \eqref{zeroentry}.

\begin{proposition}[The limit as $\gep\to 0$]\label{epsto0}
 For any $f=\sum a_j z^j\in \anp$, the limit $T_G(f)$ exists, and
 \begin{equation}\label{tgexpansion}
 T_G(f)= \frac{1}{2}\Tr\big(
\partial_{\eps\eps} f(M+\eps B)\vert_{\eps=0}\big)=
\sum_{j=0}^\infty a_jT_G(z^j).
\end{equation}
Moreover, for integers $j\ge 0$ we have
$$
T_G(z^j)=|G|^{-1/2}\, \frac{j}{2} \sum_{k_1+k_2+2=j} \Tr (B
M^{k_1}B M^{k_2}).
$$
\end{proposition}
\begin{proof}
We use the fact (see, for example, Theorem 6.2.8 in
\cite{HJ}) that for $f$ an analytic function with radius of
convergence $r>0$, and for  $A$ in the set $\mathcal M_r$
of matrices of some fixed dimension and with spectral
radius less than $r$, the power series $f(A)$ is absolutely
convergent and is analytic as a function of (the entries
of) $A$. Thus $\Tr f(M+\eps B)$ is an analytic function of
all of its (at most) $2|G|^2+1$ variables near $M,B$, and
$\eps=0$. We are free to rearrange its absolutely
convergent multiple power series expansion in any way we
like. So the first derivative  with respect to $\eps$ at
$0$ is the coefficient of the $\eps$ terms; this is a
multiple of
$$ \sum_{k_1+k_2=j-1} \text{Tr} (M^{k_1} B M^{k_2})=\sum_{k_1+k_2= j-1}
\text{Tr}
(B M^{k_1+k_2})  $$ and each of the summands is zero by
\eqref{zeroentry}. We have
$$
|G|^{1/2}\int f\, d m_{\gep} = \eps^{-2}(\Tr f(M+\eps
B)-\Tr f(M)) \to \frac{1}{2}\,\partial_{\eps\eps} \Tr
f(M+\eps B)|_{\eps=0}
$$
when $\eps\to 0$, since the first derivative vanishes at
$\eps=0$. Now the right hand side equals the
coefficient of $\eps^2$ in the expansion. This
gives \eqref{tgexpansion}.

Since the trace of a product does not change when we cyclically
permute the factors, we see that
$$
\partial_{\eps} \Tr( (M+\eps B)^{j})=\Tr( \partial_{\eps} (M+\eps B)^{j})= j
\Tr (B (M+\eps B)^{j-1}).
$$
Taking another derivative, we get the second claim of the proposition.
\end{proof}

\section{Properties of the bilinear form $H_G$} \label{covexpression}

%The goal of this section is to describe the actual and limiting
%covariance structure of the functions $T_G(f)$.
The goal of this section is to define the  bilinear form $H_G$ introduced in \eqref{H} and establish its continuity properties.

Recall that $\anp$ denotes the set of power series centered
at zero, and with radius of convergence more than 1. For a
power series $f$, let $[x^k]f(x)$ denote the coefficient of
$x^k$ in the expansion of $f$. Given a finite or infinite
graph $G$ (with $B$ and $M$) as in the introduction, we fix
a vertex $o$ of $G$, and re-define the bilinear form on
$\anp$ as
 \begin{align}\label{hg}
H_G(f,g):&= \frac{1}{2} \sum_{ij} i j \alpha_{ij}\,[x^i]f(x)[x^j]g(x) \\
\notag &=\frac{1}{2} \sum_{ij} \alpha_{ij}
\,[x^{i-1}]f'(x)[x^{j-1}]g'(x),
\end{align}
where
$$
\alpha_{ij}:=\ev \left[\sum_{v\in G\setminus
\{o\}}Y_i(G,o,v)Y_j(G,o,v)\right],
$$
and for any vertices $v, w$ of $G$, we define
\begin{equation}\label{Y}Y_j(G,v,w):=\sum_{j_1+j_2+2=j} (B M^{j_1})_{v,w} (B
M^{j_2})_{w,v}.
\end{equation}

The last sum is always finite and symmetric in $v,w$.  In
Proposition \ref{hplemma} we will show that this definition
agrees with \eqref{H} for the functions $p_\lambda$. Note
that the definition of $H_G$ does not depend on $o$ if $G$
is vertex-transitive. For any graph $G$, the bilinear form
$H_G(f,g)$ clearly makes sense for polynomials $f,g$. To go
beyond polynomials, we need to show that the infinite sum
in \eqref{hg} is well-defined. For $f=\sum a_kz^k \in
\anp$, set
$$\snorm{f}=\sum_{k=1 }^{+\infty} |a_k| k^2,$$
which defines a norm on $\mathcal X$. (To be precise, it is a norm on $\mathcal X$ modulo the constants;  however, we note that for the purposes of this paper the constant terms for functions in $\mathcal X$ do not matter, so we may thing of  $\mathcal X$ itself as the space of functions modulo the constants.)
Note that on $\anp$ this norm is finite.
\begin{remark}\label{domination}
On the subspace of $\anp$ consisting of power series with
radius of convergence strictly larger that a fixed $r>1$,
the norm $\|\cdot\|_*$ is dominated by the supremum norm on
$D_r:=\{z:|z|\le r\}$. Indeed, for $f(z)=\sum_{k=0}a_k z^k$
in that subspace, Cauchy's inequalities give $|a_k|\le
r^{-k} \|f\|_{D_r}^{\infty}$. So that
$$\|f\|_*\le \frac{r(1+r)}{(r-1)^3}\|f\|_{D_r}^\infty.$$
\end{remark}
As a direct consequence, we get
\begin{lemma}\label{dense}
Polynomials are  $\|\cdot\|_*$ - dense in $\anp$.
\end{lemma}

We are now ready to show that $H_G$ is well-defined and
continuous.

\begin{lemma} \label{uec}
 For $f,g\in \anp$ and any $G$,
the sum giving $H_G(f,g)$ is absolutely convergent, and
  $$ |H_G(f,g)| \le c_1^4 \snorm{f} \snorm {g}. $$
\end{lemma}
\begin{proof}
Note that since $|B_{vw}|\le c_1 M_{vw}$, we have
\begin{equation}\label{yjbound}
\sum_{w\in G}|Y_j(G,v,w)|\le c_1^2 \sum_{w\in
G}\sum_{j_1+j_2+2=j} M^{j_1+1}_{v,w} M^{j_2+1}_{w,v}=c_1^2
(j-1) (M^j)_{v,v}\le c_1^2 j (M^j)_{v,v}.
\end{equation}
In particular, the same bound holds for each term
$|Y_j(G,v,w)|$. Therefore
$$
\frac {1}{c_1^4ij} \sum_{v\in G\setminus
\{o\}}Y_i(G,o,v)Y_j(G,o,v) \le \sum_{v\in G\setminus
\{o\}}M^{i}_{v,v}M^{j}_{v,v} \le \sum_{v\in G\setminus
\{o\},w\in G}M^{i}_{v,w}M^{j}_{w,v} = M^{i+j}_{v,v}\le1.
$$
The claim now follows by summing over all $i,j$.
\end{proof}
\begin{proposition}\label{hplemma}  $H_G$ satisfies \eqref{H}.  Moreover,
\eqref{H} uniquely defines $H_G$ as a
$\|\cdot\|_*$-continuous bilinear form.
\end{proposition}

\begin{proof}
The absolute convergence of the series defining $H_G$,
implies
$$
H_G(p_\lambda,p_\mu)=\frac{1}{2}\sum_{i, j\ge 1} i j
\,\alpha_{i j}\lambda^{i}\mu^{j}=\frac{1}{2}\lambda\mu\,
\partial_\lambda\partial_\mu \sum_{i, j\ge 1} \alpha_{ij}
\lambda^{i}\mu^{j}.
$$
Using the definition of $\alpha_{i j}$, we write the sum
above as
\begin{eqnarray*}
\lefteqn{ \sum_{\substack{v\in G\setminus\{o\}
\\k,\ell,k',\ell'\ge 0}}\lambda^{k+\ell+2}\mu^{k'+\ell'+2}
 \ev (B M^{k})_{o v} (B M^{\ell})_{v o} (B M^{k'})_{ov }(B M^{\ell\,'})_{v
 o}}&&\\
 &=&\lambda^2\mu^2\sum_{v\in G\setminus\{o\}}\ev
(B\G_\lambda)_{ov}(B\G_\lambda)_{vo}(B\G_\mu)_{ov}(B\G_\mu)_{vo}\\
 &=&\lambda^2\mu^2\sum_{v\in G\setminus\{o\}}\ev\left[
(B\G_\lambda)_{ov}(B\G_\mu)_{ov}\right]\ev\left[(B\G_\lambda)_{vo}(B\G_\mu)_{vo}\right].
\end{eqnarray*}
This proves the first part. For the second, it suffices to
show that the linear span of
$\{p_\lambda\,:\,|\lambda|<1\}$ is $\|\cdot \|_*$-dense in
$\mathcal X$; this will be done in the next lemma.
\end{proof}

\begin{lemma}\label{dense2}
The vector space generated by $\{p_\lambda\,:\,|\lambda|<1\}$ is
$\|\cdot\|_*$ - dense in $\anp$.
\end{lemma}

\begin{proof}
Because of the density of polynomials in $\anp$ and Remark
\ref{domination}, it suffices to prove that any polynomial
$P$ can be approximated in the supremum norm on a disk
$D_r$ with $r>1$ by elements of the linear span $\{p_\lambda\,:\,|\lambda|<1\}$.

Pick $1<r<r_1$, and
let $C_{r_1}:\{z:|z|=r_1\}$. Then
$$ P(z)=\frac{1}{2\pi i} \int_{C_{r_1}} \frac{P(\zeta)}{\zeta-z}
d\zeta=\frac{1}{2\pi i} \int_{C_{r_1}}
\frac{P(\zeta)}{\zeta} p_{1/\zeta} (z) d\zeta.$$
Call $R_n(z)$ the Riemann sum corresponding to an
equipartition of the circle with $n$ pieces. The sequence
of functions $R_n$ is equicontinuous on $\{z:|z|\le r\}$
and it converges pointwise to $P(z)$, thus the convergence
is uniform on that set. So that for a given $\gd>0$, there
is a finite linear combination
$$ A_\gd(z):=c_1 p_{1/\zeta_1}(z)+c_2 p_{1/\zeta_2}(z)+\ldots+c_k
p_{1/\zeta_k}(z)$$
such that
\begin{equation} \label{close}
|f(z)-A_\gd(z)|\le \gd
\end{equation}
for all $z\in\mathbb{C}$ with $|z|\le r$. This completes the proof of the lemma.
\end{proof}

Finally, we check that $H_G(f,g)$ is continuous in $G$ as
well. Recall the definition of local convergence of graphs
from the introduction. For the following lemma, we use, as
usual, the assumptions from the introduction for each $G_n$
and $G$.
\begin{lemma} \label{graphConv}
If $G_n\to G$ locally, then for $f,g\in \anp$ we have
$H_{G_n}(f,g)\to H_G(f,g)$.
\end{lemma}
\begin{proof}
Note that as $n\to \infty$, the neighborhood of radius
$i+j$ of $o$ in $G_n$ stabilizes to look like the same
neighborhood in the limit graph $G$.

For polynomials $f, g$, and for all large $n$, we have $H_{G_n}(f,g)=H_G(f,g)$. This is because $H_{G_n}(f,g)$
only depends on a neighborhood of the root $o$ of radius
given by the maximal degree of $f$ and $g$.

Now we have that the sequence of functions $H_{G_n}(f,g)$
is equicontinuous on $\anp^2$ (by Lemma \ref{uec}) and
converges on a dense set. Thus, by Lemma \eqref{analysis}
they converge on the entire set to a continuous limit.
$H_G(\cdot, \cdot)$ is continuous (Lemma \ref{uec}), and
this finishes the proof.
\end{proof}

\section{The covariance structure} \label{s.covariance}

After establishing some properties of the bilinear form
$H_G$, we show that for finite graphs $G$ it gives the
covariance structure of $T_G(f)$.

\begin{lemma} \label{polynomCov} For any finite
vertex-transitive graph $G$ and complex polynomials $f,g$, we have
\begin{align*}
\ev(T_G(f))&=0,\\
\ev(T_G(f)T_G(g))&=H_G(f,g).
\end{align*}
\end{lemma}
\begin{proof}
We will show this for monomials. The extension to polynomials is
straightforward from bilinearity.
 First, we have
\begin{eqnarray}
T_G(z^j)&=& \frac{1}{\sqrt {|G|}}\frac{j}{2} \;\Tr\sum_{j_1+j_2+2=j}
BM^{j_1} B M^{j_2} \nonumber
\\&=&  \frac{1}{\sqrt {|G|}}\frac{j}{2} \sum_{\substack{v,w\in G \\ v\not=w}}
Y_j(G,v,w). \label{tgyj}
\end{eqnarray}
Note that the  $v=w$ terms  vanish by \eqref{zeroentry}.
Each $Y_j(G,v,w)$ with $v\ne w$ has zero mean because
different rows of $B$ are independent, with entries having
zero mean. Thus, $T(z^j)$ has also zero mean. Again because
of independence of rows of $B$,  when we compute second
moments, the terms in the sum below with $\{v, w\} \ne
\{v', w'\}$ vanish. That is
 \begin{eqnarray*}
 \EE (T_G(z^i)T_G(z^j))  &=&
\frac{ij}{4|G|}\sum_{\substack{v,w,v',w'\in G\\v\ne w,v'\ne w'}}
\EE Y_i(G,v,w) Y_j(G,v',w')
\\& =&
\frac{ij}{4|G|}\sum_{v,w\in G:v\ne w} \EE Y_i(G,v,w) \left[Y_j(G,v,w)+Y_j(G,w,v)\right]
\\& =&
 \frac{ij}{2}  \sum_{v\in G\setminus\{o\}} \EE Y_i(G,o,v) Y_j(G,o,v)
 \end{eqnarray*}
For the last equality we fixed a vertex $o$ of $G$ and used
the vertex-transitivity of the graph and the symmetry of $Y$ in its last two parameters.
\end{proof}

\begin{lemma} \label{uec2}
For $G$ finite, $|T_G(f)|\le  c_1^2 |G|^{1/2}\snorm{f}$.
\end{lemma}
\begin{proof}
This follows directly from \eqref{tgexpansion},
\eqref{tgyj}, and \eqref{yjbound}.
\end{proof}
We are now ready to prove Theorem \ref{t2}.
%%
%\begin{corollary} \label{pseriesCov} For any finite vertex-transitive graph $G$
% and $f,g\in\anp$, we have
%$$
%\ev(T_G(f)\, T_G(g))=H_G(f,g).
%$$
%\end{corollary}
%
\begin{proof}[Proof of Theorem \ref{t2}] Lemma \ref{uec} and Lemma \ref{uec2} show that  $\ev [T_G(\cdot)T_G(\cdot)]$ and
$H_G(\cdot,\cdot)$ are $\|\cdot\|_*$-continuous in $f, g$.
We conclude the proof by approximating $f$ and $g$ with
polynomials and using Lemma \ref{dense} and Lemma
\ref{polynomCov}.
\end{proof}

\section{Asymptotic normality} \label{s.normality}

For this section we consider a sequence of transitive
graphs $G_n$ converging to locally to  limit $G$. The
assumptions from the introduction hold for all $G_n$ and
the limit $G$ with the same constant $c_1$. We may also
assume that all graphs have degree $d$.

We will prove convergence of certain sequences  to normal
under a topology which we introduce now.

Let $\disdis$ denote the space of probability measures on
$\R$ with finite second moment, equipped with the
2-Wasserstein distance, call it $d_2$ (see \cite{VI},
Chapter 6). For two measures $\mu, \nu$ in the space, this
is the minimal $L^2$-distance over all possible couplings
of them, i.e.,
$$d_2(\mu, \nu):=\inf_{k}\Big(\int |x-y|^2\ dk(x,y)\Big)^{1/2},$$
where the infimum is over the set of probability measures
on $\R^2$ with first marginal $\mu$, and second $\nu$.

The space $(\disdis, d_2)$ is complete, and its topology is
stronger than weak convergence. It is easy to show that for
a sequence $X_n$ of random variables, $X_n\to X$ in this
topology (i.e., the corresponding laws converge) if and
only if $X_n\rightarrow X$ weakly and $\ev X_n^2 \to \ev
X^2$. In particular, in this topology the function
variance, $\Var:\disdis\to \R$, is continuous.

We will need the following case of Lemma 2.4 of Chen and Shao
(2004). It is a normal approximation theorem for dependent
variables. It is proved with the use of Stein's method.

\begin{lemma}[Chen and Shao (2004)] \label{SteinLemma}

Let $\C{I}$ be an index set, $\{X_i:i\in \C{I}\}$ a family of
random variables, and for $A\subset \C{I}$, let $X_{A}: =\{X_i:
i\in A\}$. Assume that

\begin{itemize}

\item[(1)] For each $i\in\C{I}$, there exist $A_i\subset B_i
\subset C_i\subset \C{I}$ such that $X_i$ is independent of
$X_{\C{I}\setminus A_i}$, $X_{A_i}$ is independent of
$X_{\C{I}\setminus B_i}$, and $X_{B_i}$ is independent of
$X_{\C{I}\setminus C_i}$.

\item[(2)] There exists a constant $\gamma$  so that for all $i\in\C{I}$ we have  $$\max(|N(C_i)|, |\{j:i\in
C_j\}|)\le \gamma,$$ where $$N(C_i):=\{j\in \C{I}: C_i \cap B_j\ne
\emptyset\}.$$

\item[(3)] Each $X_i$ has zero mean and finite variance, and
$W:=\sum_{i\in\C{I}} X_i$ satisfies $\Var(W)=1$.

\end{itemize}

Then for $2<p\le 3$, we have

$$\sup_{z\in\BBn{R}}|F(z)-\Phi(z)|\le 75 \gamma^{p-1} \sum_{i\in \C{I}}
\ev |X_i|^p,$$

where $F, \Phi$ are the distribution functions of $W$ and of the
standard normal $N(0, 1)$.
\end{lemma}

\noindent An immediate consequence is the following convergence result.

\begin{lemma} \label{normal}
For any real polynomial $f$, the sequence $T_{G_n}(f)$
converges as $n\to\infty$ in the 2-Wasserstein distance to a normal random
variable with zero mean and variance $H_G(f,f)$.
\end{lemma}

\begin{proof}

% If the polynomial is of degree at most one, then $T_{G_n}(f)=0$ for all $n$ due to the expression in Proposition \ref{epsto0}, and $H_G(f,f)=0$ because $Y_1(v,w)=0$ for all vertices $v,w$.

% We may assume therefore that  $f(x)=\sum_{j=0}^k a_j
% z^j$, with $k\ge2$ and $a_k\ne 0$.

We will apply Lemma \ref{SteinLemma}. Let   $f(x)=\sum_{j=0}^k a_j z^j$, and $\C{I}$ the set of vertices of $G_n$. In the
following, we will omit the subscripts for the matrices $M, B$
associated with the graph $G_n$. We have
$$T_{G_n}(f)=\frac{1}{\sqrt{|G_n|}} \Tr\left( \sum_{j=0}^k
 \frac{j}{2} a_j \sum_{k_1+k_2+2=j}B M^{k_1} B
M^{k_2}\right)=\sum_{v\in G_n} Y_{n, v},$$
where
$$Y_{n, v}:=\frac{1}{\sqrt{|G_n|}} \sum_{j=0}^k \frac{j}{2} a_j
\sum_{k_1+k_2+2=j}(B M^{k_1} B M^{k_2})_{v, v}.$$

It holds $\lim_{n\to+\infty} \Var\ T_{G_n}(f)=H_G(f,f)$ by Lemma
\ref{graphConv}. If this limit is zero, then the sequence $T_{G_n}(f)$ will converge to $\delta_0$ in the 2-Wasserstein topology, and the result is proved. We may therefore assume that the limit is positive, and $\Var\ T_{G_n}(f)>0$ for all $n$.

For $v\in \C{I}$, define
$$X_v:=\frac{Y_{n,v}}{\Var\ T_{G_n}(f)},\qquad W:=\,\sum_{v\in G_n}
X_v\,=\frac{T_{G_n}(f)}{\Var\ T_{G_n}(f)},
$$
and the sets
\begin{align*}
A_v:&=\{w: \text{dist} (v, w)\le k\},\\
B_v:&=\{w: \text{dist} (v, w)\le 2k\},\\
C_v:&=\{w: \text{dist} (v, w)\le 3k\}.
\end{align*}
These sets satisfy the conditions of Lemma \ref{SteinLemma}
(because the $X_v$'s corresponding to vertices that are distance
at least $k+1$ apart are independent random variables), and for
all $v$, we have $ |N(C_v)| \le|\{w:\text{dist}(v, w)\le 5
k\}|\le  d(d-1)^{5 k-1}$, and $ |\{w: v \in C_w\}|=|C_v| \le
d(d-1)^{3k-1}$. [Here we used the regularity of the graphs, but of course a uniform bound on the degree of the vertices of all graphs would work in the same way.] Also, the $X_v$'s have zero mean and finite
variance by Lemma \ref{polynomCov}. Pick any $p\in (2,3]$. Lemma
\ref{SteinLemma} gives that the distribution functions $F_W, F_Z$
of $W$ and of the standard $N(0,1)$ satisfy
\begin{equation}\label{supnorm}
\sup_{x\in\BBn{R}}|F_W(x)-F_Z(x)|\le C |G_n|^{1-\frac{p}{2}} \ev |A|^p
\end{equation}
with $C:= 75 (d(d-1)^{5 k-1})^{p-1}$ and
$$A:=\sum_{j=0}^k \frac{j}{2} a_j \sum_{k_1+k_2+2=j}(B M^{k_1} B
M^{k_2})_{o, o},$$
which is independent of $n$ for large $n$ because the sequence
converges locally to $G$. Thus relation \eqref{supnorm}, $|G_n|\to\infty$, and $p>2$ imply that the sequence
$$\frac{T_{G_n}(f)}{\Var\ T_{G_n}(f)}$$
converges to a standard normal random variable. By Theorem \ref{t2} and Lemma \ref{graphConv} we have $\Var
\ T_{G_n}(f)\to H_G(f,f)$, and the result follows. \end{proof}

%\section{Asymptotic normality for analytic functions}
%
%We are in the same setting as in the previous section, and

Our next goal is to strengthen Lemma \ref{normal} to get
Theorem \ref{normality} by showing that its conclusion is
true also for all functions $f$ in $\anp$. The proof is
simply by approximation. We will use the fact that the set
of polynomials is $\|\cdot\|_*-$ dense in $\anp$ (Lemma
\ref{dense}) and the following simple lemma (Lemma 38,
Chapter 7 of \cite{Royden}).

\begin{lemma}\label{analysis} Let $X$ be a metric space, and $Y$ a complete metric space.
Assume that $f_n:X\to Y$ is an equicontinuous sequence of
functions that converge pointwise on a dense subset of $X$.
Then the sequence $f_n$ converges pointwise on the entire
$X$ to a continuous limit.
\end{lemma}

We are ready to prove Theorem  \ref{normality}.
%
%\begin{proposition} \label{normalps}
%For $f$ a power series with real coefficients and radius of
%convergence $>1$, the
%sequence of distributions $(T_{G_n}(f))_{n\ge1}$ converges in
%2-Wasserstein distance  to a normal random variable with variance
%$H_G(f,f)$.
%\end{proposition}

\begin{proof}[Proof of Theorem \ref{normality}]
Let $\hat T_n(f)$ denote the distribution of $T_{G_n}(f)$. The
sequence of functions $\hat T_n:\anp \to \disdis$ is uniformly
equicontinuous because, by Lemma \ref{uec} and Theorem \ref{t2} we have
$$
\ev |T_{G_n}(f)-T_{G_n}(g)|^2\le  c_1^4 \|f-g\|_*^2.$$ By Lemma
\ref{normal}, they converge pointwise at polynomials, which
form a dense subset of $\anp$ by Lemma \ref{dense}. By
Lemma \ref{analysis}, the limit $\hat T(f)$ of $\hat
T_n(f)$ exists for all functions $f\in\anp$ and is
continuous. Also, for the 2-Wasserstein distance between
$\hat T(f)$ and $\hat T(g)$, we have
\begin{equation}\label{WLips} d_2(\hat T(f),
\hat T(g))\le c_1^2 \|f-g\|_*.
\end{equation}
Since the limit is normal on a dense set of points, and
limits of normal random variables are normal, it follows
that all limits $\hat T (f)$ are normal. Also the
functionals
$$f\mapsto \Var \hat
T(f),\  f\mapsto H_G(f, f)$$
are $\|\cdot\|_*$- continuous (the first because of
\eqref{WLips} and a property of the 2-Wasserstein distance,
the second by Lemma \ref{uec}) and they agree on a
$\|\cdot\|_*$-dense set by Lemma \ref{normal}. Thus, $\Var
\hat T(f)=H_G(f,f)$ for all $f\in\anp$.
\end{proof}

\section{A formula for the covariance}

The goal of this section is to reduce the problem of
computing the covariance kernel $\beta$ of Theorem
\ref{main} to inverting a Stieltjes transform.

When the limit graph $G$ is bipartite (see the discussion
before Theorem \ref{main}), we would like to write  the
limiting covariance in the form
\begin{equation}\label{bip}
H_G(f\circ s,g\circ s) = \int f'(x)g'(y) d\beta(x,y),
\end{equation}
 where $s(x)=x^2$. It is
sufficient to check the identity \eqref{bip} for the functions
$f=p_{\lambda^2}$ and $g=p_{\mu^2}$, since their linear
span is a dense subset in $\mathcal X$ with respect to the
norm $\|\cdot \|_*$ (Lemma \ref{dense2}), and the
integral on the right is clearly continuous with respect to the
product topology based on this norm.  For this choice of
$f,g$ the right hand side of \eqref{bip} equals
$$
 \int \frac{\gl^2}{(1-\gl^2 x)^2}\frac{\mu^2}{(1-\mu^2 y)^2}  \,d \gb(x,y)=\frac{1}{4}
 \gl \mu \, \partial_{\gl ,\mu}^2 \int
 \frac{1}{(x-\gl^{-2})}\frac{1}{(y-\mu^{-2})} \,d \gb(x,y).
 $$
Note that $p_\lambda-p_{\lambda^2}\circ s$ is an odd
function. For bipartite graphs $G$ it is clear from the
expressions \eqref{hg} and \eqref{Y} that $H_G$ vanishes if
one its arguments is odd. So we have
$$
H_G(p_{\lambda^2}\circ s, p_{\mu^2}\circ s) =
H_G(p_{\lambda}, p_{\mu}).
$$
In light of the expression \eqref{H} we arrive at the
following:
\begin{proposition} If the Stieltjes transform relation
\begin{equation}\label{bRelation}
 \gl^2\mu^2 \sum_{v\not=o} 2 \ev \Big[(B \mathcal G_\gl)_{o,
v} (B \mathcal G_\mu)_{o,v} \Big] \ev \Big[(B \mathcal
G_\gl)_{v, o} (B \mathcal G_\mu)_{v,o} \Big] =
\int\frac{1}{(x-\gl^{-2})}\frac{1}{(y-\mu^{-2})} \,d
\gb(x,y)
\end{equation}
holds for all $|\lambda|,|\mu|<1$,  then \eqref{bip} holds
for all $f,g\in \mathcal X$.
\end{proposition}

\section{Explicit formulas for trees} \label{s.trees}

In this section, we look at the case where the limiting graph $G$
is the infinite $d$-regular tree $\BB{T}_d$, we compute the measure
$d\gb$ introduced in \eqref{bip}, and thus complete the
proof of Theorem \ref{main}.

From now on we define square roots for complex numbers
with
a branch cut at the negative real axis. More precisely, for
$z=|z|e^{i \theta}$, with $\theta\in(-\pi, \pi]$, we set
$\sqrt{z}=\sqrt{|z|}e^{i\theta/2}$.

\begin{proof}[Proof of the covariance formula in Theorem
\ref{main}]

We start by computing the left hand side of
\eqref{bRelation} explicitly. The Green's function for the
infinite $d$-regular tree,
$\G_\lambda(v,w)=\G_\lambda(\dist(v,w))$, is given by
$$
\G_\lambda(r)=b_\lambda a_\lambda^r,
$$
where
$$a_\gl:=d\, \frac{1 -\sqrt{1-\rho \gl^2}}{2(d-1) \gl},\ \
b_\gl:=(1-\gl a_\gl)^{-1},$$
and $\rho=4(d-1)/d^2$ as in the introduction. See \cite{WO}, Lemma
1.24.

First, an observation. For two vertices $i, j$, let $\ell$ denote
the unique neighbor of $i$ closest to $j$, $r$ the distance between $i$
and $j$, and define $B_{i j}^*=B_{i \ell}$. For $\gl$ with modulus
less than 1, by symmetry and since $\sum_k B_{ik}=0$, we get
\begin{eqnarray*}
(B\G_\gl)_{ij}=B_{i\ell}\G_\gl(r-1)+
 \sum_{\substack{k\sim i \\ k\ne\ell}}
 B_{ik}\G_\gl(r+1)= B_{ij}^* \G^*_\lambda(r),
\end{eqnarray*}
where
\begin{align*}
 \G^*_\lambda(r)&=\G_\lambda(r-1)-\G_\lambda(r+1)=b_\lambda(1-a_\lambda^2)a_\lambda^{r-1}.
\end{align*}
 Thus for any vertex $w\ne o$, with $r=\dist(o,w)$, we have
 $$
 \ev (B
 \G_\gl)_{o, w} (B \G_\mu)_{o,w}= \ev B_{ow}^{*2} \G^*_\lambda(r)
 \G^*_\mu(r) =\G^*_\lambda(r) \G^*_\mu(r),
 $$
because $\ev {B_{o, w}^*}^2=1$, and therefore the sum in
the left hand side of \eqref{bRelation} equals
\begin{equation}\label{aboveratio}
 \sum_{w\in \BB{T}^d\setminus\{o\}}2
\G^*_\lambda(r)^2\G^*_\mu(r)^2 = 2d\sum_{r=1}^\infty
(d-1)^{r-1}
   \,\G^*_\lambda(r)^2 \G^*_\mu(r)^2
 =
2  \frac{ (b_\lambda b_\mu)^2(1- a_\lambda^2)^2(1-
a_\mu^2)^2d}{1-(d-1)(a_\lambda a_\mu)^2}.
 \end{equation}
Now  this can be expressed in terms of $s=\sqrt{1-\rho\gl^2}$
and $t=\sqrt{1-\rho \mu^2}$. Indeed, in these variables, we have the
simpler expressions
$$
a_\lambda^2=\frac{1-s}{(d-1)(1+s)},\qquad
b_\lambda=\frac{2(d-1)}{(d-2)+sd},
$$
and \eqref{aboveratio} becomes
\begin{equation}\label{st}\frac{32\, (d-1) d}
  {(1+s)(1+t)\left(
 (d - 2)(1 + s) (1 + t) + 2 (s + t)   \right) }.
 \end{equation}
Introduce new variables $u,v$ by $u=\gl^{-2}, v=\mu^{-2}$,
and let $\hat\gb(u,v)$  denote expression \eqref{st} as a
function of them. In terms of these variables, relation
\eqref{bRelation} becomes
$$
\frac{\hat \gb(u,v)}{uv}=
\int\frac{1}{(x-u)}\frac{1}{(y-v)} \,d \gb(x,y).
$$
Due to our convention for square roots (see beginning of this section), the quantity $s$ is an analytic function of $u$ in $\BB{C}\setminus [0, \rho]$, and the same holds for $t$ as a function of $v$.

Assume that $d>2$. Then, the denominator in \eqref{st} does
not vanish because $s,t$ have nonnegative real parts (if we
set the last factor in the denominator equal to zero and
solve for $t$, we get a quantity with negative real part).
So  $\hat \beta$ defined by \eqref{st} is a holomorphic
function of $(u, v)$ on $(\BB{C}\setminus [0, \rho])^2$. In
fact, even the limits of the denominator when $u$ or $v$
approaches $[0, \rho)$ are not zero.

Since the function $h(u,v):=(u v)^{-1}\hat \beta(u, v)$ is
holomorphic in $(\mathbb{C}\setminus [0, \rho])^2$, and it
decays as $(uv)^{-1}$ near infinity, Cauchy's formula will
express its values in terms of double contour integrals
around the segment $[0, \rho]$. Shrinking the contour
around $[0, \rho]$, we get a line integral, and we take
into account the different limits of  $h$ as one of its
arguments approaches the segment from the upper or the
lower half plane. That is, when $u$ approaches
$x\in[0,\rho)$ from the upper half plane, we have $s(u)\to
s(x)$. While when $u$ approaches $x$ from the lower half
plane, we have $s(u)\to -s(x)$. The difference comes from
the branch cut discontinuity of square root in the
definition of $s$. Thus we have
 $$
\frac{\hat\beta(x, y)}{xy}=-\frac{1}{4\pi^2}\int_0^\rho\int_0^\rho
\frac{1}{(u-x)(v-y)} \frac{ \tilde\beta(u, v)\,du\,dv}{uv},
 $$
where for $u, v\in [0,\rho]$ we have
\begin{eqnarray*}\tilde \beta(u,v)&=& \sum_{\sigma,\tau=\pm 1}
\sigma\tau \hat \beta[\sigma s,\tau t] \\
&=&\frac{512(d-1)^2 d^2 (d-2) \,st}
{(d^2(s^2+t^2)-(d-2)^2(1+s^2 t^2))^2-(8 (d-1) s t)^2},
\end{eqnarray*}
and $\hat \beta[\sigma s,\tau t]$ refers to the expression
\eqref{st} with $(s,t)$ replaced by $(\sigma s, \tau t)$.

Consequently, the density of the measure $d\beta$ is
$$\beta_d(u, v)=-\frac{1}{4\pi^2}\frac{\tilde \beta(u,v)}{u v}.$$
Substituting the expressions of $s, t$ in terms of $u, v$,
%
%The claim of the lemma
%now follows from changing variables from $(s,t)$ to $(u,v)$.
%
we find
$$\beta_d(u, v)=
\frac{128}{\pi^2}d^2(d-1)^2 (d-2) \frac{\sqrt{u(\rho-u)}
\sqrt{v(\rho-v)}}{A},$$
with
\begin{align*}
A:&=u^2 v^2 \bigg((d^2(s^2+t^2)-(d-2)^2(1+s^2 t^2))^2-(8
(d-1) s t)^2\bigg)\\
%&=\rho^4 (d-2)^4+8\rho^2\bigg((4d-6)(u-v)^2+(d-2)^2 \big(\frac{d-1}{2d^2} \kappa(\frac{u+v}{2})^2+\frac{d+3}{2}(u-v)^2\big)\bigg)\\
&=\rho^2\left[16(2d-3)(u-v)^2+(d-2)^2\bigg(\rho
\kappa\big(\frac{u+v}{2}\big)^2+4(d+3)(u-v)^2+\rho^2
(d-2)^2\bigg)\right],
\end{align*}
where $\kappa(u)=2d\sqrt{u(\rho-u)}$. Then
 $$
 \beta_d(u, v)=\frac{2\,\kappa(u)\kappa(v)\,(d-2)\,d^4\,\pi^{-2}}{16(2d-3)(u-v)^2+(d-2)^2\big(\rho k\big(\frac{u+v}{2}\big)^2+4(d+3)(u-v)^2+\rho^2
 (d-2)^2\big)}.
$$
 This proves the $d>2$ case.

The case $d=2$ can be easily shown by using continuity in the
formulas as $d\downarrow 2$. We get
$$
\beta_2(x, y)=\frac{8}{\pi}\,\kappa(x)
\delta_x(y)=\frac{32}{\pi}\sqrt{x(1-x)}\  \delta_x(y).
$$
The qualitative difference here is
that for $d=2$ the denominator of \eqref{st} does vanish along a
line.
\end{proof}

\bigskip

\noindent {\bf Acknowledgments.}  This research is supported by the Sloan and Connaught grants,
the NSERC discovery grant program, and the Canada Research Chair program (Vir\'ag). We thank Amir Dembo for encouraging discussions. 

\bibliography{spectrum}
\bigskip
\bigskip

\noindent  Dimitris Cheliotis. Eindhoven University of Technology, {\sc Eurandom.}  L.G. 1.26,
   P.O. Box 513,
   Eindhoven,  5600 MB,
  The Netherlands. {\tt dimitrisc@gmail.com}
 
 \bigskip
  
\noindent B\'alint Vir\'ag. University of Toronto,
40 St George St.
Toronto, ON, M5S 2E4, Canada.
{\tt balint@math.toronto.edu}

\end{document}